\documentclass[12pt,reqno]{amsart}
\usepackage{amsmath,amsthm,amssymb,amsfonts,amscd}
\usepackage{mathrsfs}
\usepackage{bbding}
\usepackage{graphicx,latexsym}
\usepackage{hyperref} %latex建立参考文献的超链接

\numberwithin{equation}{section} % 定义公式按节编号

\theoremstyle{plain}
\newtheorem{theorem}{Theorem}
\newtheorem{lemma}{Lemma}

\newtheorem{proposition}{Proposition}

\theoremstyle{definition}

\theoremstyle{remark}
\newtheorem{remark}{Remark}

\begin{document}

\title{Exponential sums over primes in short intervals \\ 
and an application to the Waring--Goldbach problem}
\author{Bingrong Huang}
\address{School of Mathematics \\ Shandong University \\ Jinan \\Shandong 250100 \\China}
\email{brhuang@mail.sdu.edu.cn}
\date{\today}

\begin{abstract}
  Let $\Lambda(n)$ be the von Mangoldt function, $x$ real and $2\leq y \leq x$.
  This paper improves the estimate on the exponential sum over primes in short intervals
  \[
    S_k(x,y;\alpha) = \sum_{x< n \leq x+y} \Lambda(n) e\left( n^k \alpha \right)
  \]
  when $k\geq 3$ for $\alpha$ in the minor arcs.
  And then combined with the Hardy--Littlewood circle method, this enables us to investigate the Waring--Goldbach problem of representing a positive integer $n$ as the sum of $s$ $k$th powers of almost equal prime numbers, which improves the results in Wei and Wooley \cite{wei2014sums}.
  %some short interval variants of Hua's theorems in additive prime number theory.
\end{abstract}


\maketitle
\tableofcontents

\section{Introduction} \label{sec: introduction}

Let $\Lambda(n)$ be the von Mangoldt function, $k\geq 1$ an integer, $x$ real and $2 \leq y \leq x$. The estimate of the exponential sum
over primes in short intervals
\begin{equation}
  S_k(x,y;\alpha) = \sum_{x < n \leq x+y} \Lambda(n) e\left(n^k \alpha\right)
\end{equation}
was first studied by I. M. Vinogradov \cite{vinogradov1939estimation} in 1939 with his elementary method. Since then this topic has attracted the interest of quite a number of authors. These sums arise naturally and play important roles when solving the Waring--Goldbach problems in short
intervals by the circle method. In particular, the case $k = 1$, i.e., the linear exponential sum over primes in short intervals, was studied quite extensively, because of its applications to the study of the Goldbach--Vinogradov theorem with three almost equal prime variables (see \cite{zhan1991representation} and the references therein).

For the case $k=2$, Liu and Zhan  \cite{liu1999estimation} first established a non-trivial estimate of $S_2(x,y;\alpha)$ for all $\alpha$ and all published results before their result are valid only for $\alpha$ in a very thin subset of $[0,1]$. In \cite{lv2007exponential}, L\"u and Lao improved the results in \cite{liu1999estimation} to be as good as what was previously derived from the Generalized Riemann Hypothesis.

In this paper we deal with $S_k(x,y;\alpha)$ in the general case $k\geq 3$. Let $y = x^\theta$ with $3/4 < \theta \leq 1$.
We set
\begin{equation} \label{eqn: Q}
  P = x^{2K\delta}, \quad \textrm{and} \quad  Q = x^{k-2} y^2 P^{-1},
\end{equation}
where
\begin{equation} \label{eqn: K & t_k}
  K = 2 t_k(t_k+2), \quad  t_k = k(k-1)
\end{equation}
are defined as in \cite[equations (1.2) and (2.1)]{wei2014sums}
and $\delta$ is a positive parameter which may depend on $k$.
By Dirichlet's theorem on Diophantine approximations, every real number $\alpha$ has a rational approximation $a/q$, where $a$ and $q$ are integers subject to
\begin{equation} \label{eqn: Dirichlet's theorem on Diophantine approximations}
  1\leq q\leq Q, \quad  (a,q)=1, \quad  |q\alpha-a|\leq 1/Q.
\end{equation}
We denote by $\mathfrak{M}$ the union of the \emph{major arcs}
\begin{equation*}
  \mathfrak{M}(q,a) = \{ \alpha\in [0,1) : |q\alpha-a|\leq Q^{-1} \},
\end{equation*}
with $0\leq a\leq q\leq P$ and $(a,q)=1$, and let $\mathfrak{m} = [0,1)\setminus\mathfrak{M}$ for the set of \emph{minor arcs} complementary to the set $\mathfrak{M}$.
% 讲点历史
In Liu and Zhan \cite{liu1996estimation} and Huang and Wang \cite{huang2015exponential}, they handle with $S_k(x,y;\alpha)$ for all $\alpha \in [0,1]$ by dividing this into several parts and then estimating them by different methods.
%In Liu, L\"{u} and Zhan \cite{liu2006exponential}, Kumchev \cite{kumchev2013weyl} and Wei and Wooley \cite{wei2014sums}, the authors mainly bounded $S_k(x,y;\alpha)$ for $\alpha \in \mathfrak{m}$.
%Liu, L\"{u} and Zhan \cite{liu2006exponential} used methods from multiplicative number theory, and then Kumchev \cite{kumchev2013weyl} used the sieve method to improve their results.
The work of Liu, L\"{u} and Zhan \cite{liu2006exponential} is focused on the major arc behavior (hence, the use of multiplicative methods) and provides a major arc estimate that complements the ``true minor arc bound'' from an earlier paper of Liu and Zhan \cite{liu1996estimation}.
%%%
Kumchev \cite{kumchev2013weyl} and Wei and Wooley \cite{wei2014sums} focus on the minor arcs.
The method of Kumchev \cite{kumchev2013weyl} is a variant of his earlier results on long sums \cite{kumchev2006weyl} that avoids the use of sieve methods and relies on Heath-Brown's identity instead.
%%%
Recently, Wei and Wooley \cite{wei2014sums} gives a substitute for a Weyl-type estimate for $S_k(x,y;\alpha)$, which makes use of Daemen's estimates in \cite{daemen2010asymptotic}
%which based on his use of the so-called \emph{binomial descent method}.
via a bilinear form treatment motivated by analogous arguments making use of Vinogradov's mean value theorem.
They establish an estimate of Weyl-type that delivers non-trivial
estimates throughout the set of minor arcs $\mathfrak{m}$ provided that $\theta$ is a real number with $5/6<\theta<1$.
We first state our main result for exponential sums.
%%%%%%%%%%%%%%%%%%%%%%%%%%%%%%%%%%%%%%%%%%%%%%%%%%%%%%%%%%%%%%%%%%
%%                      主要定理                                %%
%%%%%%%%%%%%%%%%%%%%%%%%%%%%%%%%%%%%%%%%%%%%%%%%%%%%%%%%%%%%%%%%%%
\begin{theorem} \label{thm: exponential sums}
  Let $k\geq 3$. Let $\theta$ be a real number with $3/4 < \theta \leq 1$ and suppose that $0< \rho < \rho_k(\theta)$, where
  \begin{equation} \label{eqn: rho}
    \rho_k(\theta) = \min \left\{ \frac{\sigma_k (\theta - 3/4)}{8}, \delta \right\}, % \frac{\theta - 2/3}{2}
  \end{equation}
  with $\sigma_k = 1/(2t_k)$.
  Then, for any fixed $\varepsilon > 0$, we have
  \begin{equation}
    S_k(x,y;\alpha) \ll  y^{1-\rho +\varepsilon} + \frac{y x^\varepsilon}{(q + y^2 x^{k-2} |q\alpha - a|)^{1/(2k)}}.
  \end{equation}
\end{theorem}

\begin{remark}
  In \cite{huang2015strong}, we asserts the strong orthogonality between the M{\"o}bius function and nonlinear exponential functions in short intervals. More specifically, if $k \geq 3$ being fixed and $y\geq x^{1-1/4+\varepsilon}$, then for any $A>0$, we have
  \[
    \sum_{x< n \leq x+y} \mu(n) e\left( n^k \alpha \right) \ll y(\log y)^{-A}
  \]
  uniformly for $\alpha \in \mathbb{R}$.
\end{remark}

To prove Theorem \ref{thm: exponential sums}, we first use Vaughan's identity to divide this to two types of sums, and then estimate the exponential sums of type I and type II respectively (see \S \ref{sec: proof of Theoerem 1}). The exponential sums of type I will be estimated in \S \ref{sec: Type I}, and the other one will be treated in \S \ref{sec: Type II}. To do these we mainly follow the method of Kumchev \cite{kumchev2013weyl}, and combine with the results in Daemen \cite{daemen2010asymptotic} (see \S \ref{sec: auxiliary results}-\ref{sec: Type II}).

\medskip
In the last section, we apply the circle method to give an application of this exponential sum %over primes in short intervals
to Waring--Goldbach problem in short intervals.

A formal application of the circle method suggests that whenever $s$ and $k$ are natural numbers with $s\geq k+1$, then all large integers $n$ satisfying appropriate local conditions should be represented as the sum of $s$ $k$-th powers of prime numbers.
With this expectation in mind, consider a natural number $k$ and prime $p$, take $\tau = \tau(k,p)$ to be the integer with $p^\tau|k$ but $p^{\tau+1}\nmid k$, and the define $\eta = \eta(k,p)$ by putting $\eta(k,p) = \tau+2$, when $p=2$ and $\tau > 0$, and otherwise $\eta(k,p) = \tau +1$. We then define $R= R(k)$ by putting $R(k) = \prod p^\eta$, where the product is taken over primes $p$ with $(p-1)|k$.
Write $X = (n/s)^{1/k}$.
We say that the exponent $\Delta_{k,s}$ is \emph{admissible} when, provided that $\Delta$ is a positive number with $\Delta < \Delta_{k,s}$, then for all sufficiently large positive integers $n$ with $n \equiv s \pmod{R}$, the equation
\begin{equation} \label{eqn: equation}
  p_1^k + p_2^k + \cdots + p_s^k = n
\end{equation}
has a solution in prime numbers $p_j$ satisfying $|p_j - X|\leq X^{1-\Delta} \ (1\leq j \leq s)$. We refer the reader to \cite{wei2014sums} for more details.

Together the circle method used in \cite{wei2014sums} with our estimate of exponential sums in Theorem \ref{thm: exponential sums}, we show that there are larger admissible exponents $\Delta_{k,s}$ as soon as $s> 2t_k$.

%%%%%%%%%%%%%%%%%%%%%%%%%%%%%%%%%%%%%%%%%%%%%%%%%%%%%%%%%%%%%%%%%%
%%          小区间上的 Waring--Goldbach 问题                     %%
%%%%%%%%%%%%%%%%%%%%%%%%%%%%%%%%%%%%%%%%%%%%%%%%%%%%%%%%%%%%%%%%%%
\begin{theorem} \label{thm: sums of powers of almost equal primes}
  Let $s$ and $k$ be integers with $k\geq 3$ and $s >2 t_k$. Suppose that $\varepsilon>0$, that $n$ is a sufficiently large number satisfying $n \equiv s \pmod{R}$, and write $X = (n/s)^{1/k}$. Then the equation $n = p_1^k + p_2^k + \cdots + p_s^k$ has a solution in prime numbers $p_j$ with $|p_j - X| \leq X^{19/24+\varepsilon} \ (1\leq j \leq s)$.
\end{theorem}

% 讨论一下这个结果的强度
In \cite{wei2014sums}, Wei and Wooley gave the same exponent $\frac{19}{24}$ for the case $k=2$.
They remark that since $\frac{19}{24} = \frac{1}{2} \left(1+ \frac{7}{12}\right)$, this exponent is in some sense half way between the trivial exponent 1 and the exponent $\frac{7}{12}$ that, following the work of Huxley \cite{huxley1972difference}, represents the effective limit of our knowledge concerning the asymptotic distribution of prime numbers in short intervals.

%%%%%%%%%%%%%%%%%%%%%%%%%%%%%%%%%%%%%%%%%%%%%%%%%%%%%%%%%%%%%%%%%%
%%   以下为本文的第三个结果，Waring--Goldbach 问题中的例外集     %%
%%%%%%%%%%%%%%%%%%%%%%%%%%%%%%%%%%%%%%%%%%%%%%%%%%%%%%%%%%%%%%%%%%
\medskip
By the same argument in \cite[\S 9]{wei2014sums}, we obtain the following almost-all result. (The history of this kind of problem can be seen in \cite[\S 1]{wei2014sums}.)

\begin{theorem}\label{thm: exceptional sets}
  Let $s$ and $k$ be integers with $k\geq 3$ and $s > t_k$. Suppose that $\varepsilon>0$. Then for almost all positive integers $n$ with $n \equiv s \pmod{R}$, (and, in case $k=3$ and $s=7$, satisfying also $9\nmid n$), the equation $n = p_1^k + p_2^k + \cdots + p_s^k$ has a solution in prime numbers $p_j$ with $|p_j - X| \leq X^{19/24+\varepsilon} \ (1\leq j \leq s)$, where $X = (n/s)^{1/k}$.
\end{theorem}

%another kind of Waring--Goldbach problem
%
%Denote by $R_3(N,U)$ the number of solutions of the equation
%\begin{equation*}
%   \left\{\begin{array}{ll}
%     N = p_1 + p_2 + p_3^k, \\
%     |p_1 - \frac{N}{3}|\leq U,\ |p_2 - \frac{N}{3}|\leq U,\ |p_3^k - \frac{N}{3}|\leq U.
%   \end{array}\right.
%\end{equation*}
%
%\begin{theorem} \label{thm: another kind of Waring--Goldbach problem}
%  Let
%\end{theorem}

%\smallskip
\noindent
{\bf Notation.}
Throughout the paper, the letter $\varepsilon$ denotes a sufficiently small positive real number which may be different at each occurrence. For example, we may write $x^\varepsilon \ll y^\varepsilon$.
Any statement in which $\varepsilon$ occurs holds for each positive $\varepsilon$, and any implied constant in such a statement is allowed to depend on $\varepsilon$.
The letter $p$, with or without subscripts, is reserved for prime numbers.
In addition, as usual, $e(z)$ denotes $e^{2\pi iz}$.
We write $(a, b) = {\rm gcd}(a, b$), and we use $m \sim M$ as an abbreviation for the condition $M < m \leq 2M$.

%%%%%%%%%%%%%%%%%%%%%%%%%%%%%%%%%%%%%%%%%%%%%%%%%%%%%%%%%%%%%%%%%%
%%                一些引理                                      %%
%%%%%%%%%%%%%%%%%%%%%%%%%%%%%%%%%%%%%%%%%%%%%%%%%%%%%%%%%%%%%%%%%%
\section{Auxiliary results} \label{sec: auxiliary results}

The following lemma is useful to give an estimate for exponential sums of type I and type II which is an improvement of \cite[Lemma 2.2]{kumchev2013weyl}.

\begin{lemma} \label{lemma: exponential sums in short intervals}
  Let $k\geq 3$ be an integer and $\gamma\geq 3$ be a real number. Let $0<\rho\leq \sigma_k/\gamma$, where $\sigma_k = 1/(2t_k)$.
   % $\sigma_k = \max\left\{ \frac{1}{2^{k-1}}, \frac{1}{2k(k-2)} \right\}$.
  Suppose that $y\leq x,$ and $y \geq x^{\frac{\gamma}{2\gamma-\sigma_k-1}}$. Then either    % $\mathcal{I}$ denote the subinterval of $(x,x+y]$
  \begin{equation} \label{eqn: exponential sums in short intervals, minor arcs}
    \sum_{x < n \leq x+y} e\left(n^k \alpha\right) \ll y^{1-\rho +\varepsilon},
  \end{equation}
  or there exist integers $a$ and $q$ such that
  \begin{equation} \label{eqn: exponential sums in short intervals, conditions}
    1\leq q \leq y^{k\rho},\quad (a,q)=1,\quad |q\alpha - a| \leq x^{1-k} y^{k\rho -1},
  \end{equation}
  and
  \begin{equation} \label{eqn: exponential sums in short intervals, major arcs}
    \sum_{x < n \leq x+y} e\left(n^k \alpha\right)  \ll   y^{1-\rho +\varepsilon} + \frac{y}{(q+yx^{k-1}|q\alpha - a|)^{1/k}}.
  \end{equation}
\end{lemma}

% We reproduce the proof appearing in \cite{huang2014strong}.

\begin{proof}
  Take
  \begin{equation*}
    P_0 = y^{1/\gamma} \quad  \textrm{and} \quad Q_0 = x^{k-2}y^2/P_0.
  \end{equation*}
  By Dirichlet's theorem on Diophantine approximation, there exists integers $a$ and $q$ with
  \begin{equation} \label{eqn: conditions of a and q}
    1\leq q\leq Q_0, \quad  (a,q)=1, \quad  |q\alpha-a|\leq 1/Q_0.
  \end{equation}
%%%%
  When $q>P_0$, we rewrite the sum on the left of (\ref{eqn: exponential sums in short intervals, minor arcs}) as
  \begin{equation*}
    \sum_{1\leq n\leq z} e(\alpha_k n^k + \alpha_{k-1} n^{k-1} + \cdots + \alpha_0),
  \end{equation*}
  where $z\leq y$ and $\alpha_j = \binom{k}{j} \alpha u^{k-j}$, with $u$ a fixed integer.
  Hence, it follows from the argument underlying the proof of \cite[eq. (3.5)]{daemen2010asymptotic} and \cite[eq. (4.23)]{wei2014sums} that
  \begin{equation}
    \sum_{x < n \leq x+y} e\left(n^k \alpha\right) \ll y P_0^{-1/(2t_k)+\varepsilon} \ll y^{1-\rho +\varepsilon}.
  \end{equation}
%%%%
  When $q\leq P_0$, from \cite[eq. (5.1)-(5.5) and \S 6]{daemen2010asymptotic}, we deduce
  \begin{equation*}
    \sum_{1\leq n\leq y} e(\alpha_k n^k + \alpha_{k-1} n^{k-1} + \cdots + \alpha_0) \ll \frac{y}{(q+yx^{k-1}|q\alpha - a|)^{1/k}} + \Delta,
  \end{equation*}
  where
  \begin{equation*}
    \Delta \ll P_0^{1/2+\varepsilon} \left( 1+ \frac{P_0 x^k}{x^{k-2} y^2} \right)^{1/2} \ll P_0^{1+\varepsilon} x/y \ll y^{1-\rho +\varepsilon},
  \end{equation*}
  provided that $y \geq x^{\frac{\gamma}{2\gamma-\sigma_k-1}}$. Thus, at least one of (\ref{eqn: exponential sums in short intervals, minor arcs}) and (\ref{eqn: exponential sums in short intervals, major arcs}) holds.
  The lemma follows on noting that when conditions (\ref{eqn: exponential sums in short intervals, conditions}) fail, inequality (\ref{eqn: exponential sums in short intervals, minor arcs}) follows from (\ref{eqn: exponential sums in short intervals, major arcs}).
\end{proof}

%\begin{remark}
%  Essentially Best result!
%\end{remark}

The next lemma gives some inequalities which will be used in the following sections.

\begin{lemma} \label{lemma: some inequalities}
  We have
  \begin{equation}
    \sum_{n\sim N} (r,n^k)^{1/k} \leq  N \tau(r),
  \end{equation}
  where $\tau(r)$ is the divisor function; and for any $\varepsilon > 0$, we have
  \begin{equation} \label{eqn: sum with R(n,h)}
    \sum_{\substack{n\sim N \\ (n,h)=1}} (r,R(n,h))^{1/k} \ll N r^{\varepsilon} + r^{1/k+\varepsilon},
  \end{equation}
  where $R(n,h) = \left((n+h)^k - n^k \right)/h$.
\end{lemma}

\begin{proof}
  We have
  \begin{equation*}
    \sum_{n\sim N} (r,n^k)^{1/k} \leq \sum_{n\sim N} (n,r) \leq \sum_{d|r} d \sum_{\substack{n\sim N \\ d|n}} 1 \leq N \tau(r).
  \end{equation*}
  To prove (\ref{eqn: sum with R(n,h)}), see the inequality (3.11) in Kawada and Wooley \cite{kawada2001waring}.
\end{proof}

\section{Type I estimate} \label{sec: Type I}

The following proposition treats the exponential sums of type I
which is an improvement of \cite[Lemma 8]{huang2015exponential}

\begin{proposition} \label{prop: Type I}
  Let $k\geq 3$ be an integer and $\gamma\geq 3$ be a real number.
  Let $0 < \rho < \min\{\sigma_k /(2\gamma),\delta\}$, with $\sigma_k = 1/(2t_k)$. % and $\rho < \frac{1}{k^3}$.
  % Let $M,N\geq 1$, $MN \asymp x$.
  Suppose that $\alpha$ is real that there exist integers $a$ and $q$ such that
  (\ref{eqn: Dirichlet's theorem on Diophantine approximations}) holds with $Q$ given
  by (\ref{eqn: Q}). Let $\xi(m) \ll m^\varepsilon$, and define
  \begin{equation*}
    \mathcal{T}_1 = \sum_{m\sim M}  \xi(m) \sum_{x<mn\leq x+y} e\left((mn)^k \alpha\right).
  \end{equation*}
  Then
  \begin{equation*}
    \mathcal{T}_1 \ll y^{1-\rho +\varepsilon} + \frac{y x^\varepsilon}{(q + y x^{k-1} |q\alpha - a|)^{1/k}},
  \end{equation*}
  provided that
  \begin{equation} \label{eqn: conditions of M. type I}
    M \ll y \left( \frac{y}{x} \right)^{\frac{\gamma}{\gamma-\sigma_k-1}}, \quad
    % M \ll y^{\frac{2d-\sigma_k}{d-\sigma_k-1}} x^{-\frac{d+1}{d-\sigma_k-1}}, \quad
    M \ll y x^{-\gamma\rho/\sigma_k}, \quad
    M^{2k} \ll y x^{k-1-2k\rho}.
  \end{equation}
\end{proposition}

\begin{proof}
  Set \[  S_m = \sum_{X < n \leq X+Y} e\left( m^k n^k \alpha \right), \]
  where $X = x/m, Y =y/m$, with $m\sim M$.
  Define $\nu$ by $Y^\nu = x^\rho L^{-1}$. Note that, by (\ref{eqn: conditions of M. type I}), we have
  $$ \nu < \sigma_k/\gamma. $$
  %%%
  We denote by $\mathcal{M}$ the set of integers $m \sim M$, for which there exist integers $b_1$ and $r_1$ with
  \begin{equation} \label{eqn: conditions of r_1 and b_1}
    1 \leq r_1 \leq Y^{k\nu},\quad (b_1,r_1)=1,\quad  |r_1 m^k \alpha - b_1| \leq X^{1-k} Y^{k\nu -1}.
  \end{equation}
  %%%
  We apply Lemma \ref{lemma: exponential sums in short intervals} to the summation over $n$ and get
  \begin{equation*}
    S_m \ll  Y^{1-\nu+\varepsilon} +  \frac{Y}{(r_1+ YX^{k-1}|r_1 m^k\alpha- b_1|)^{1/k}},
  \end{equation*}
  for $m \in \mathcal{M}$.
  So
  \begin{eqnarray*}
    \mathcal{T}_1 &\ll&  \sum_{m\sim M} |\xi(m)| Y^{1-\nu+\varepsilon}  + \sum_{m\in \mathcal{M}} \frac{|\xi(m)| Y}{(r_1+ YX^{k-1}|r_1 m^k\alpha- b_1|)^{1/k}}.
  \end{eqnarray*}
  Then, we have
  \[
    \mathcal{T}_1 \ll y^{1-\rho +\varepsilon} + T_1(\alpha),
  \]
  where
  \[
    T_1(\alpha) = \sum_{m\in \mathcal{M}} \frac{|\xi(m)| Y}{(r_1+ YX^{k-1}|r_1 m^k\alpha- b_1|)^{1/k}}.
  \]
  We apply Dirichlet's theorem on Diophantine approximation to find integers $b$ and $r$ with
  \begin{equation} \label{eqn: conditions of r and b}
    1\leq r \leq x^{-k\rho} Y X^{k-1}, \quad (b,r)=1, \quad  |r\alpha -b| \leq x^{k\rho} Y^{-1} X^{1-k}.
  \end{equation}
  By (\ref{eqn: conditions of M. type I}), (\ref{eqn: conditions of r_1 and b_1}) and (\ref{eqn: conditions of r and b}), we have
  \begin{eqnarray*}
    |b_1 r - bm^k r_1| &=& |r(b_1-r_1m^k\alpha) + r_1 m^k (r\alpha - b)| \\
     &\leq& x^{-k\rho} Y X^{k-1}  X^{1-k} Y^{k\nu -1} + Y^{k\nu} (2M)^k x^{k\rho} Y^{-1} X^{1-k} \\
     % &\leq& x^{-k\rho} Y X^{k-1}  x^{k\rho}L^{-k} X^{1-k} Y^{-1} +  \\
     &\ll& L^{-k} +  M^{2k} L^{-k}  x^{2k\rho-k+1} y^{-1} \ll L^{-k} < 1,
  \end{eqnarray*}
  whence
  \[
    \frac{b_1}{r_1} = \frac{m^k b}{r}, \quad  r_1 = \frac{r}{(r,m^k)}.
  \]
  Thus, by Lemma \ref{lemma: some inequalities}, we have
  \begin{eqnarray*}
    T_1(\alpha) &\leq&  \sum_{m\in \mathcal{M}} \frac{|\xi(m)| y M^{-1} r_1^{-1/k}}{(1+ YX^{k-1}|m^k\alpha- m^k b/r|)^{1/k}} \\
     &\ll&  \frac{y M^{-1+\varepsilon}}{(1+ y x^{k-1} |\alpha- b/r|)^{1/k}}  \sum_{m\sim M} \left(\frac{r}{(r,m^k)}\right)^{-1/k} \\
     &\ll&  \frac{y x^\varepsilon}{(r + yx^{k-1} |r \alpha- b|)^{1/k}}.
  \end{eqnarray*}
  Recall that $b$ and $r$ satisfy the conditions (\ref{eqn: conditions of r and b}). We now consider three cases depending on the size of $r$ and $|r\alpha -b|$.
  \begin{description}
    \item[\it Case 1] If $r > x^{k\rho} $, then $T_1(\alpha) \ll y^{1-\rho+\varepsilon}$;
    \item[\it Case 2] If $r \leq  x^{k\rho}$ and $|r\alpha -b| > y^{-1} x^{1-k} x^{k\rho}$, then $T_1(\alpha) \ll y^{1-\rho+\varepsilon}$;
    \item[\it Case 3] If $r \leq  x^{k\rho}$ and $|r\alpha -b| \leq  y^{-1} x^{1-k} x^{k\rho}$, we have
        \begin{eqnarray*}
          |ra-bq| &=& |r(a-q\alpha) + q(r\alpha-b)| \\
            &\leq& x^{k\rho}\frac{1}{Q} + Q y^{-1} x^{1-k} x^{k\rho} \\
            &\leq&  \frac{x^{k\rho}P}{x^{k-2}y^2} + \frac{yx^{k\rho}}{xP}.
        \end{eqnarray*}
        Since $\rho < \delta$, we have $|ra-bq|<1$, hence
        $$ a=b, \quad q=r.$$
        Then
        $$
          T_1(\alpha) \ll  \frac{y x^\varepsilon}{(q + yx^{k-1} |q \alpha- a|)^{1/k}}.
        $$
  \end{description}
  So we prove
  \begin{equation*}
    \mathcal{T}_1 \ll y^{1-\rho +\varepsilon} + \frac{y x^\varepsilon}{(q + y x^{k-1} |q\alpha - a|)^{1/k}}.
  \end{equation*}
\end{proof}

\begin{remark}
  Let
  \begin{equation} \label{eqn: Type I with log}
    \mathcal{T}_1^{*} = \sum_{m\sim M} \xi(m) \sum_{x<mn\leq x+y}  e\left((mn)^k \alpha\right) \log n.
  \end{equation}
  Under the condition of Proposition \ref{prop: Type I} we have
  \[
    \mathcal{T}_1^{*} \ll y^{1-\rho +\varepsilon} + \frac{y x^\varepsilon}{(q + y x^{k-1} |q\alpha - a|)^{1/k}}.
  \]
\end{remark}

\begin{remark}
  One can estimate the exponential sums
  \begin{equation*}
    \sum_{m_1\sim M_1} \sum_{m_2\sim M_2} \xi(m_1,m_2) \sum_{x<m_1m_2n\leq x+y} e\left((m_1m_2n)^k \alpha\right)
  \end{equation*}
  with some suitable conditions on $M_1$ and $M_2$ as \cite[Lemma 3.2]{kumchev2013weyl} and \cite[Lemma 4.2]{wei2014sums} did, and then may give a better result than Proposition \ref{prop: Type I}. Since it has no influence on our main results, we will not do it.
\end{remark}

\section{Type II estimate} \label{sec: Type II}

To prove Theorem \ref{thm: exponential sums}, we also need to handle the exponential sums of type II.
Let $\xi(m)$ and $\eta(n)$ be arithmetic functions satisfying the property that for all natural numbers $m$ and $n$, one has
\begin{equation}
  \xi(m) \ll m^\varepsilon \quad \textrm{and} \quad  \eta(n) \ll n^\varepsilon.
\end{equation}
Let $M$ and $N$ be positive parameters, and define the exponential sum
$\mathcal{T}_2 = \mathcal{T}_{2}(\alpha;M)$ by
\begin{equation}
  \mathcal{T}_2(\alpha;M) :=  \sum_{M<m\leq 2M} \xi(m) \sum_{ x< mn\leq x+y} \eta(n) e\left( (mn)^k \alpha \right).
\end{equation}
%$\mathcal{T}_2 = \mathcal{T}_{2}(\alpha;M,N)$ by
%\begin{equation}
%  \mathcal{T}_2(\alpha;M,N) = \max_{M'\leq 2M} \max_{N'\leq 2N} \Bigg| \sum_{M<m\leq M'} \xi(m) \sum_{\substack{N<n\leq N' \\ x< mn\leq x+y}} b(n) e((mn)^k \alpha) \Bigg|.
%\end{equation}
The following proposition gives an estimate for $\mathcal{T}_2$ which is an improvement of \cite[Lemma 3.1]{kumchev2013weyl}

\begin{proposition} \label{prop: Type II}
  Let $k,\gamma,\sigma_k$ be as in Proposition \ref{prop: Type I}. Let $0 < \rho < \min\{\sigma_k /(8\gamma), \delta\}$.
  Suppose that $\alpha$ is real that there exist integers $a$ and $q$ such that (\ref{eqn: Dirichlet's theorem on Diophantine approximations}) holds with $Q$ given by (\ref{eqn: Q}).
  And let $x$ and $y$ be positive numbers with
  \begin{equation} \label{eqn: conditions of y. type II}
    y = x^\theta, \quad  \frac{1}{(1-2\rho)} \frac{3\gamma-\sigma_k-1}{2(2\gamma-\sigma_k-1)} \leq \theta \leq 1.
  \end{equation}
  Then
  \begin{equation*}
    \mathcal{T}_2 \ll y^{1-\rho +\varepsilon} + \frac{y x^\varepsilon}{(q + y^2 x^{k-2} |q\alpha - a|)^{1/(2k)}},
  \end{equation*}
  provided that
  \begin{equation} \label{eqn: conditions of M. type II}
    x^{1/2} \leq M \ll y^{1-2\rho}.
  \end{equation}
\end{proposition}

\begin{proof}
  % 先定义参数
  Set $N = x/M$, $X=x/N$ and $Y = y/N = yM/x$. Define $\nu$ by $Y^\nu = x^{2\rho}L^{-1}$.
  By (\ref{eqn: conditions of M. type II}), we have
  $$
    \nu < \sigma_k/\gamma.
  $$
  For $n_1,\ n_2 \leq 2N$, let
  \begin{equation*}
    \mathcal{M}(n_1,n_2) = \{ m\in (M,2M] : x < mn_1, mn_2 \leq x+y \}.
  \end{equation*}
  % 转化成估计 $T_1(\alpha)$
  By Cauchy's inequality and an interchange of the order of summation, we have
  \begin{equation} \label{eqn: mathcal{T}_2 to T_1}
    \mathcal{T}_2^2 \ll y^{1+\varepsilon} M + M x^\varepsilon T_1(\alpha),
  \end{equation}
  where
  \begin{equation*}
    T_1(\alpha) = \sum_{n_1<n_2} \left| \sum_{m\in \mathcal{M}(n_1,n_2)} e\left( \alpha (n_2^k - n_1^k) m^k \right) \right|.
  \end{equation*}
  % 考虑集合 $\mathcal{N}$
  Let $\mathcal{N}$ denote the set of pairs $(n_1,n_2)$ with $n_1<n_2$ and $\mathcal{M}(n_1,n_2)\neq \varnothing$ for which there exist integers $b$ and $r$ such that
  \begin{equation} \label{eqn: conditions of r and b. Type II}
    1\leq r \leq Y^{k\nu}, \quad (b,r)=1, \quad |r(n_2^k-n_1^k)\alpha - b|\leq Y^{k\nu-1}X^{1-k}.
  \end{equation}
  % 现在开始转化为 $T_2(\alpha)$
  Since $N/2< n_1 < n_2 \leq 2N$ and $\mathcal{M}(n_1,n_2)\neq \varnothing$, we have $n_2 - n_1 \leq yx^{-1}n_1$. Hence $\#\mathcal{N} \ll xyM^{-2}$. In order to handle the inner summation in $T_1(\alpha)$, we set
  \begin{eqnarray*}
    X_1 &=& \max\left\{ M,\frac{x}{n_1} \right\} \asymp M = \frac{x}{N} = X, \\
    Y_1 &=& \min\left\{ 2M,\frac{x+y}{n_2} \right\} - \max\left\{ M,\frac{x}{n_1} \right\} \ll \frac{y}{N} = Y.
  \end{eqnarray*}
  If $Y_1 < X_1^{\gamma/(2\gamma-\sigma_k-1)}$, by (\ref{eqn: conditions of y. type II}) and (\ref{eqn: conditions of M. type II}), the contribution to $T_1(\alpha)$ is
  \begin{equation*}
    \ll  xy M^{-2} M^{\gamma/(2\gamma-\sigma_k-1)} \ll  y^{2-2\rho +\varepsilon} M^{-1}.
  \end{equation*}
  If $Y_1 \geq X_1^{\gamma/(2\gamma-\sigma_k-1)}$, since $\nu<\sigma_k/d$, we can apply Lemma \ref{lemma: exponential sums in short intervals} with $\rho=\nu$, $x=X_1$ and $y=Y_1$ to the inner summation in $T_1(\alpha)$. We get
  \begin{equation} \label{eqn: T_1 to T_2}
    T_1(\alpha) \ll  y^{2-2\rho +\varepsilon} M^{-1} + T_2(\alpha),
  \end{equation}
  where
  \begin{equation*}
    T_2(\alpha) = \sum_{(n_1,n_2)\in \mathcal{N}} \frac{Y}{\left( r + Y X^{k-1} |r(n_2^k - n_1^k)\alpha - b| \right)^{1/k}}.
  \end{equation*}

  We now change the summation variables in $T_2(\alpha)$ to
  \begin{equation*}
    d= (n_1,n_2), \quad n=n_1/d, \quad h = (n_2-n_1)/d.
  \end{equation*}
  We obtain
  \begin{equation} \label{eqn: T_2}
    T_2(\alpha) \ll \sum_{dh\leq y/M} {\sum_n}' \frac{Y}{\left( r + Y X^{k-1} |rh d^k R(n,h)\alpha - b| \right)^{1/k}},
  \end{equation}
  where $R(n,h) = ((n+h)^k - n^k)/h$ and the inner summation is over $n$ with $(n,h)=1$ and $(nd,(n+h)d)\in \mathcal{N}$.
  For each pair $(d,h)$ appearing in the summation on the right side of (\ref{eqn: T_2}), Dirichlet's theorem on Diophantine approximation yields integers $b_1$ and $r_1$ with
  \begin{equation} \label{eqn: conditions of r_1 and b_1. Type II}
    1\leq r_1 \leq x^{-2k\rho} YX^{k-1}, \quad (b_1,r_1) = 1, \quad |r_1 hd^k\alpha - b_1| \leq x^{2k\rho} Y^{-1}X^{1-k}.
  \end{equation}
  As $R(n,h) \leq 4^k (N/d)^{k-1}$, combining (\ref{eqn: conditions of M. type II}), (\ref{eqn: conditions of r and b. Type II}) and (\ref{eqn: conditions of r_1 and b_1. Type II}), we have
  \begin{equation*}
    \begin{split}
      | b_1 r R(n,h) - br_1 | & = |rR(n,h)(b_1-r_1hd^k\alpha) + r_1(rhd^kR(n,h)\alpha-b)| \\
        &\leq  r_1 Y^{k\nu-1}X^{1-k} + r R(n,h) x^{2k\rho} Y^{-1}X^{1-k} \\
        &\leq  L^{-k} + 4^k N^{k-1} x^{2k\rho} L^{-k} x^{2k\rho} Y^{-1}X^{1-k} <1.
    \end{split}
  \end{equation*}
  Hence,
  \begin{equation} \label{eqn: b/r and r}
    \frac{b}{r} = \frac{b_1 R(n,h)}{r_1}, \quad r = \frac{r_1}{(r_1,R(n,h))}.
  \end{equation}
  Combining (\ref{eqn: T_2}) and (\ref{eqn: b/r and r}), we obtain
  \begin{equation*}
    T_2(\alpha) \ll \sum_{dh\leq y/M} \frac{Y}{\left( r_1 + Y X^{k-1} N_d^{k-1} |r_1 h d^k\alpha - b_1| \right)^{1/k}} \sum_{\substack{n\sim N_d \\ (n,h)=1}} (r_1,R(n,h))^{1/k},
  \end{equation*}
  where $N_d = N/d$. By Lemma \ref{lemma: some inequalities}, we deduce that
  \begin{equation} \label{eqn: T_2 to T_3}
    T_2(\alpha) \ll y^2 x^{-1+\varepsilon} + T_3(\alpha),
  \end{equation}
  where
  \begin{equation*}
    T_3(\alpha) = \sum_{dh\leq y/M} \frac{yx^\varepsilon /d}{\left( r_1 + Y X^{k-1} N_d^{k-1} |r_1 h d^k\alpha - b_1| \right)^{1/k}}.
  \end{equation*}

  We now write $\mathcal{H}$ for the set of pairs $(d,h)$ with $dh \leq y/M$ for which there exist integers $b_1$ and $r_1$ subject to
  \begin{equation} \label{eqn: restriction conditions of r_1 and b_1. Type II}
    1\leq r_1 \leq x^{2k\rho}, \quad  (b_1,r_1) = 1, \quad  |r_1hd^k\alpha - b_1| \leq x^{-k+1+2k\rho} Y^{-1}.
  \end{equation}
  We have
  \begin{equation} \label{eqn: T_3 to T_4}
    T_3(\alpha) \ll  y^{2-2\rho+\varepsilon} M^{-1} + T_4(\alpha),
  \end{equation}
  where
  \begin{equation*}
    T_4(\alpha) = \sum_{(d,h)\in \mathcal{H}} \frac{yx^\varepsilon /d}{\left( r_1 + Y X^{k-1} N_d^{k-1} |r_1 h d^k\alpha - b_1| \right)^{1/k}}.
  \end{equation*}
  % 接下来再分别去掉分母中 $\alpha$ 前面的 h 和 d
  For each $d\leq y/M$, Dirichlet's theorem on Diophantine approximation yields integers $b_2$ and $r_2$ with
  \begin{equation} \label{eqn: conditions of r_2 and b_2. Type II}
    1\leq r_2 \leq x^{k-1-2k\rho} Y/2, \quad  (b_2,r_2) = 1, \quad  |r_2 d^k \alpha - b_2|\leq 2 x^{-k+1+2k\rho} Y^{-1}.
  \end{equation}
  Combining (\ref{eqn: restriction conditions of r_1 and b_1. Type II}) and (\ref{eqn: conditions of r_2 and b_2. Type II}), we obtain
  \begin{equation*}
    \begin{split}
      |b_2r_1h - b_1r_2| & = |r_1h(b_2-r_2d^k\alpha)+r_2(r_1hd^k\alpha-b_1)| \\
          &\leq r_1 h |r_2d^k\alpha - b_2| + r_2 |r_1hd^k\alpha - b_1| \\
          &\leq 1/2 + 2 x^{-k+2+4k\rho} M^{-2} < 1,
    \end{split}
  \end{equation*}
  whence
  \begin{equation*}
    \frac{b_1}{r_1} = \frac{hb_2}{r_2}, \quad r_1 = \frac{r_2}{(r_2,h)}.
  \end{equation*}
  We write $Z_d = Y X^{k-1} N_d^{k-1} |r_2 d^k\alpha - b_2|$ and by Lemma \ref{lemma: some inequalities}, we get
  \begin{equation*}
    T_4(\alpha) = \sum_{(d,h)\in \mathcal{H}} \frac{yx^\varepsilon /d}{\left( r_1 + Z_d h \right)^{1/k}} (r_2,h)^{1/k} \ll \sum_{d\leq y/M} \frac{y^2x^\varepsilon M^{-1}}{d^2 (r_2 + y(Md)^{-1}Z_d )^{1/k}}.
  \end{equation*}
  Hence
  \begin{equation} \label{eqn: T_4 to T_5}
    T_4(\alpha) \ll y^{2-2\rho+\varepsilon}M^{-1} + T_5(\alpha),
  \end{equation}
  where
  \begin{equation*}
    T_5(\alpha) = \sum_{d\in \mathcal{D}} \frac{y^2 x^\varepsilon M^{-1}}{d^2 (r_2 + y(Md)^{-1}Z_d )^{1/k}},
  \end{equation*}
  and $\mathcal{D}$ is the set of integers $d\leq x^{2\rho}$ for which there exist integers $b_2$ and $r_2$ with
  \begin{equation} \label{eqn: restriction conditions of r_2 and b_2. Type II}
    1\leq r_2 \leq x^{2k\rho}, \quad  (b_2,r_2)=1, \quad  |r_2 d^k \alpha - b_2| \leq y^{-2} x^{2-k+2k\rho}.
  \end{equation}
  Combining (\ref{eqn: Q}), (\ref{eqn: Dirichlet's theorem on Diophantine approximations}) and (\ref{eqn: restriction conditions of r_2 and b_2. Type II}), we deduce that
  \begin{equation*}
    \begin{split}
      |r_2 d^k a - b_2 q| & = |r_2 d^k (a-q\alpha) + q(r_2 d^k \alpha - b_2)| \\
          & \leq  r_2 d^k Q^{-1} + q |r_2 d^k \alpha - b_2| \\
          & \leq  x^{4k\rho} Q^{-1} + y^{-2} x^{2-k+2k\rho} Q < 1,
    \end{split}
  \end{equation*}
  whence
  \begin{equation*}
    \frac{b_2}{r_2} = \frac{d^k a}{q} , \quad  r_2 = \frac{q}{(q,d^k)}.
  \end{equation*}
  Thus, recalling Lemma \ref{lemma: some inequalities}, we get
  \begin{equation} \label{eqn: T_5}
    T_5(\alpha) \ll  \frac{y^2 x^\varepsilon M^{-1}}{(q+ y^2 x^{k-2} |q\alpha - a|)^{1/k}}  \sum_{d\leq x^{2\rho}} \frac{(q,d^k)^{1/k}}{d^2}
    \ll  \frac{y^2 x^\varepsilon M^{-1}}{(q+ y^2 x^{k-2} |q\alpha - a|)^{1/k}} .
  \end{equation}

  The desired estimate follows from (\ref{eqn: conditions of y. type II}), (\ref{eqn: conditions of M. type II}), (\ref{eqn: mathcal{T}_2 to T_1}), (\ref{eqn: T_1 to T_2}), (\ref{eqn: T_2 to T_3}), (\ref{eqn: T_3 to T_4}), (\ref{eqn: T_4 to T_5}) and (\ref{eqn: T_5}).
\end{proof}

\section{Proof of Theorem \ref{thm: exponential sums}} \label{sec: proof of Theoerem 1}

In this section we deduce Theorem \ref{thm: exponential sums} from Propositions \ref{prop: Type I}
and \ref{prop: Type II} and Vaughan's identity for $\Lambda(n)$.

\begin{proof}[Proof of Theorem \ref{thm: exponential sums}]
  Let % We put
  \begin{equation} \label{eqn: U & V}
    U = x^{\theta/2 - \rho}, \quad  V = x^{1 - \theta + 2\rho}.
  \end{equation}
  By (\ref{eqn: rho}), we have
  \begin{equation} \label{eqn: UV & x/U}
    UV \asymp (x+y)/U \asymp x^{1-\theta/2 + \rho} \ll y^{1-2\rho}.
  \end{equation}
  And then we apply Vaughan's identity (see \cite{vaughan1980recent}) in the following form
  \begin{equation}
    \Lambda(n) = \sum_{\substack{md=n \\ 1\leq d \leq V}} \mu(d) \log m - \sum_{\substack{lmd=n \\ 1\leq d\leq V \\ 1\leq m\leq U}} \mu(d)\Lambda(m) - \sum_{\substack{lmd=n\\ 1\leq d\leq V \\ m> U \\ ld> V}}  \mu(d) \Lambda(m).
  \end{equation}
  Thus we deduce that
  \begin{equation} \label{eqn: S_k to S_1+S_2+S_3}
    S_k(x,y;\alpha) = S_1 - S_2 -S_3,
  \end{equation}
  where
  \begin{eqnarray*}
    S_1 &=& \sum_{1\leq d\leq V} \mu(d) \sum_{x< md \leq x+y} (\log m) e\left((md)^k\alpha\right), \\
    S_2 &=& \sum_{1\leq v\leq UV} \lambda_0(v) \sum_{x< lv \leq x+y} e\left((lv)^k\alpha\right),\\
    S_3 &=& \sum_{V < u\leq (x+y)/U} \lambda_1(u) \sum_{\substack{x< mu \leq x+y \\ m> U}} \Lambda(m) e\left((mu)^k\alpha\right),
  \end{eqnarray*}
  in which
  \begin{equation*}
    \lambda_0(v) = \sum_{\substack{md=v \\ 1\leq d\leq V \\ 1\leq m\leq U}} \mu(d)\Lambda(m) \quad \textrm{and} \quad  \lambda_1(u) = \sum_{\substack{d|u \\ 1\leq d \leq V}} \mu(d).
  \end{equation*}

  We begin with estimating the sum $S_3$. Take
  \begin{equation} \label{eqn: d}
    \gamma = (\theta - 3/4)^{-1}.
  \end{equation}
  Since $3/4<\theta\leq 1$, by (\ref{eqn: rho}), we have
  \begin{equation*}
    \frac{1}{(1-2\rho)} \frac{3\gamma-\sigma_k-1}{2(2\gamma-\sigma_k-1)} \leq \theta \leq 1.
  \end{equation*}
  To apply Proposition \ref{prop: Type II}, we further divide $S_3$ into to two parts
  \begin{equation*}
    S_{31} = \sum_{x^{1/2} \leq u \leq (x+y)/U} \lambda_1(u) \sum_{\substack{x< mu \leq x+y \\ m> U}} \Lambda(m) e\left((mu)^k\alpha\right),
  \end{equation*}
  and
  \begin{equation*}
    S_{32} = \sum_{V < u < x^{1/2}} \lambda_1(u) \sum_{\substack{x< mu \leq x+y \\ m> U}} \Lambda(m) e\left((mu)^k\alpha\right).
  \end{equation*}
  On noting that (\ref{eqn: U & V}), (\ref{eqn: UV & x/U}) and $\lambda_1(u) \leq \tau(u)$, we can divide the summation over $u$ into dyadic intervals to deduce from Proposition \ref{prop: Type II} that
  \begin{eqnarray*}
    S_{31} &\ll&  (\log x) \max_{x^{1/2} \leq M \leq (x+y)/U} \left| \sum_{u \sim M} \xi(u) \sum_{x< mu \leq x+y} \eta(m) e\left((mu)^k\alpha\right) \right|\\
           &\ll&  y^{1-\rho +\varepsilon} + \frac{y x^\varepsilon}{(q + y^2 x^{k-2} |q\alpha - a|)^{1/(2k)}},
  \end{eqnarray*}
  where $\xi(u) = \lambda_1(u)$, and $\eta(m) = \Lambda(m)$ if $m>U$ and is 0 if else.
  For $S_{32}$, we first interchange the order of summation, and then by the same argument as above, we obtain
  \begin{equation*}
    S_{32} \ll y^{1-\rho +\varepsilon} + \frac{y x^\varepsilon}{(q + y^2 x^{k-2} |q\alpha - a|)^{1/(2k)}}.
  \end{equation*}
  Hence we get
  \begin{equation} \label{eqn: S_3}
    S_{3} \ll y^{1-\rho +\varepsilon} + \frac{y x^\varepsilon}{(q + y^2 x^{k-2} |q\alpha - a|)^{1/(2k)}}.
  \end{equation}

  Next we estimate $S_2$. Write
  \begin{equation*}
    S_4(Z,W) = \sum_{Z < v \leq W} \lambda_0(v) \sum_{x< lv \leq x+y} e\left((lv)^k\alpha\right).
  \end{equation*}
  Then we find that
  \begin{equation} \label{eqn: S_2 to S_4}
    S_2 = S_4(0,V) + S_4(V,UV).
  \end{equation}
  Note that (\ref{eqn: U & V}), (\ref{eqn: UV & x/U}) and the bound $|\lambda_0(v)| \leq \log v$, we deduce from Proposition \ref{prop: Type II} that
  \begin{equation} \label{eqn: S_4(V,UV)}
    S_4(V,UV) \ll y^{1-\rho +\varepsilon} + \frac{y x^\varepsilon}{(q + y^2 x^{k-2} |q\alpha - a|)^{1/(2k)}}.
  \end{equation}
  We then estimate $S_4(0,V)$.
  Since $3/4<\theta\leq 1$, by (\ref{eqn: rho}), (\ref{eqn: U & V}) and (\ref{eqn: d}), we have
  \begin{equation*}
    V \ll y \left( \frac{y}{x} \right)^{\frac{\gamma+1}{\gamma-\sigma_k-1}}, \quad
    V \ll y x^{-\gamma\rho/\sigma_k}, \quad
    V^{2k} \ll y x^{k-1-2k\rho}.
  \end{equation*}
  So we can divide the summation over $v$ into dyadic intervals to deduce from Proposition \ref{prop: Type I} that
  \begin{equation} \label{eqn: S_4(0,V)}
    S_4(0,V) \ll y^{1-\rho +\varepsilon} + \frac{y x^\varepsilon}{(q + y x^{k-1} |q\alpha - a|)^{1/k}}.
  \end{equation}
  Thus, by combining (\ref{eqn: S_4(V,UV)}) and (\ref{eqn: S_4(0,V)}), we deduce from (\ref{eqn: S_2 to S_4}) that
  \begin{equation} \label{eqn: S_2}
    S_2 \ll y^{1-\rho +\varepsilon} + \frac{y x^\varepsilon}{(q + y^2 x^{k-2} |q\alpha - a|)^{1/(2k)}}.
  \end{equation}

  Finally, in order to estimate $S_1$, we apply (\ref{eqn: Type I with log}) directly, proceeding as in the treatment of $S_4(0,V)$. Thus we again obtain the bound
  \begin{equation} \label{eqn: S_1}
    S_1 \ll y^{1-\rho +\varepsilon} + \frac{y x^\varepsilon}{(q + y^2 x^{k-2} |q\alpha - a|)^{1/(2k)}}.
  \end{equation}
  Theorem \ref{thm: exponential sums} follows from (\ref{eqn: S_k to S_1+S_2+S_3}), (\ref{eqn: S_3}), (\ref{eqn: S_2}) and (\ref{eqn: S_1}).
\end{proof}

\section{Proof of Theorem \ref{thm: sums of powers of almost equal primes}} \label{sec: proof of Theorem 2}

We outline our proof of Theorem \ref{thm: sums of powers of almost equal primes}, which proceeds via the circle method. Suppose that $k$ and $s$ are integers with $k\geq 2$ and $s\geq t_k$, where $t_k$ is defined as in (\ref{eqn: K & t_k}).
Let $\theta$ be a real number with $3/4 < \theta < 1$,
and let $\delta$ be a sufficiently small, but fixed, positive number with $4K\delta < \min\{ \theta - 3/4, 1 - \theta \}$.
Consider a sufficiently large natural number $N$, put $X = (N/s)^{1/k}$, and write $Y = X^\theta$.
When $n$ is a natural number with $N\leq n \leq N+X^{k-1}Y$, we denote
\begin{equation*}
  \rho_s(n) = \sum_{|p_1-X|\leq Y} \cdots \sum_{|p_s-X|\leq Y} (\log p_1) \cdots (\log p_s),
\end{equation*}
which is the weighted number of solutions of the equation (\ref{eqn: equation}) with $|p_i - X|\leq Y\ (1\leq i \leq s)$.
Define
\begin{equation}
  f(\alpha) = \sum_{\substack{|p-X|\leq Y\\ p\ \mathrm{prime}}} (\log p) e(p^k\alpha).
\end{equation}
Then it follows from orthogonality that
\begin{equation} \label{eqn: rho_s(n)}
  \rho_s(n) = \int_0^1 f(\alpha)^s e(-n\alpha) d\alpha.
\end{equation}

Next we define the Hardy--Littlewood dissection. We rewrite $P$ and $Q$ to be
\begin{equation*}
  P = X^{2K\delta}, \quad   Q = X^{k-2}Y^2P^{-1}.
\end{equation*}
Then let $\mathfrak{M}$ and $\mathfrak{m}$ be the major arc and minor arc as in \S \ref{sec: introduction}, respectively, with $P$ and $Q$ defined above.
When $\mathfrak{B}$ is a measurable subset of $[0,1)$, we define
\begin{equation}
  \rho_s(n;\mathfrak{B}) = \int_{\mathfrak{B}} f(\alpha)^s e(-n\alpha) d\alpha.
\end{equation}
Thus, since $[0,1)$ is the disjoint union of $\mathfrak{M}$ and $\mathfrak{m}$, we find from (\ref{eqn: rho_s(n)}) that
\begin{equation}
  \rho_s(n) = \rho_s(n;\mathfrak{M}) + \rho_s(n;\mathfrak{m}).
\end{equation}

The major arc contribution can be summarised in the following proposition.

\begin{proposition} \label{prop: major arcs}
  Suppose that $k\geq 2$ and $s\geq \min\{ 5,k+2 \}$. Then, whenever $19/24 < \theta < 1$, $Y = X^\theta$ and $n$ is a natural number with $N \leq n \leq N + X^{k-1}Y$, we have
  \begin{equation}
    \rho_s(n;\mathfrak{M}) = \mathfrak{S}(n) \mathfrak{J}(n) + O(Y^{s-1}X^{1-k}(\log X)^{-1}),
  \end{equation}
  where the singular integral
  \begin{equation*}
    \mathfrak{J}(n) = \int_0^1 v(\beta)^s e(-\beta n) d\beta
  \end{equation*}
  with
  \begin{equation*}
    v(\beta) = k^{-1} \sum_{(X-Y)^k \leq m \leq (X+Y)^k} m^{-1+1/k} e(\beta m),
  \end{equation*}
  and the singular series
  \begin{equation*}
    \mathfrak{S}(n) = \sum_{q=1}^{\infty} \varphi(q)^{-s} \sum_{\substack{a=1 \\ (a,q)=1}}^{q} S(q,a)^s e(-na/q).
  \end{equation*}
  with
  \begin{equation*}
    S(q,a) = \sum_{\substack{r=1 \\ (q,r)=1}}^{q} e(a r^k/q).
  \end{equation*}
  Moreover, we have
  \begin{equation}
    Y^{s-1} X^{1-k} \ll \mathfrak{S}(n) \mathfrak{J}(n) \ll Y^{s-1} X^{1-k} (\log X)^\eta,
  \end{equation}
  where $\eta = \eta(s,k)$ is a positive number.
\end{proposition}

\begin{proof}
  See \cite[Proposition 2.1 and eq. (2.7)]{wei2014sums}.
\end{proof}

In order to estimate the minor arc contribution, we have the following analogue of Hua's lemma.

\begin{proposition} \label{prop: minor arcs, integral}
  Suppose that $y$ is a real number with $y \geq x^{1/2}$. Then whenever $s\geq 2 t_k$ and $\varepsilon > 0$, we have
  \begin{equation}
    \int_0^1 |f(\alpha)|^s d\alpha  \ll  y^{s-1} x^{1-k+\varepsilon}.
  \end{equation}
\end{proposition}

\begin{proof}
  See \cite[Proposition 2.2]{wei2014sums}.
\end{proof}

Next, by Theorem \ref{thm: exponential sums}, we establish a non-trivial estimate for $f(\alpha)$ throughout the set of minor arcs $\mathfrak{m}$.

\begin{proposition} \label{prop: exponential sums}
  Let $k\geq 3$. Let $\theta$ be a real number with $3/4 < \theta \leq 1$ and suppose that
  \begin{equation} \label{eqn: the value of rho}
    \varrho = \varrho_k(\theta) = \frac{1}{2} \min \left\{ \frac{\sigma_k (\theta - 3/4)}{8}, \delta \right\},
  \end{equation}
  where $\sigma_k = 1/(2t_k)$.
  Then, for any fixed $\varepsilon > 0$, we have
  \begin{equation}
    \max_{\alpha \in \mathfrak{m}} |f(\alpha)| \ll  Y^{1-\varrho +\varepsilon}.
  \end{equation}
\end{proposition}

\begin{proof}
  Take $x = X-Y, \ y = 2Y$. Recall that for $\alpha \in \mathfrak{m}$, we have $q > P$. Proposition \ref{prop: exponential sums} is an easy consequence of Theorem \ref{thm: exponential sums}.
\end{proof}

\begin{proof}[Proof of Theorem \ref{thm: sums of powers of almost equal primes}]
  Now following the same argument as in \cite[\S 2]{wei2014sums}, we give the proof of Theorem \ref{thm: sums of powers of almost equal primes} by combining Propositions \ref{prop: major arcs}, \ref{prop: minor arcs, integral} and \ref{prop: exponential sums}.
\end{proof}

\bigskip
\noindent
{\bf Acknowledgement.} The author would like to thank Professor Jianya Liu for his valuable advice and constant encouragement. The author also want to thank Bin Wei for explaining many details in their paper.

\bigskip
\bibliographystyle{amsplain}
\bibliography{hbrbib_exponentialsums}

\providecommand{\bysame}{\leavevmode\hbox to3em{\hrulefill}\thinspace}
\providecommand{\MR}{\relax\ifhmode\unskip\space\fi MR }
\providecommand{\MRhref}[2]{%
  \href{http://www.ams.org/mathscinet-getitem?mr=#1}{#2}
}
\providecommand{\href}[2]{#2}
\begin{thebibliography}{10}

\bibitem{daemen2010asymptotic}
D.~Daemen, \emph{The asymptotic formula for localized solutions in {W}aring's
  problem and approximations to {W}eyl sums}, Bull. Lond. Math. Soc.
  \textbf{42} (2010), no.~1, 75--82. \MR{2586968 (2011b:11139)}

\bibitem{huang2015strong}
B.~R. Huang, \emph{Strong orthogonality between the {M}{\"o}bius function and
  nonlinear exponential functions in short intervals}, Int. Math. Res. Not.
  (2015), rnv091.

\bibitem{huang2015exponential}
B.~R. Huang and Z.~W. Wang, \emph{Exponential sums over primes in short
  intervals}, Journal of Number Theory \textbf{148} (2015), 204--219.

\bibitem{huxley1972difference}
M.~N. Huxley, \emph{On the difference between consecutive primes}, Invent.
  Math. \textbf{15} (1972), 164--170. \MR{0292774 (45 \#1856)}

\bibitem{kawada2001waring}
K.~Kawada and T.~D. Wooley, \emph{On the {W}aring--{G}oldbach problem for
  fourth and fifth powers}, Proceedings of the London Mathematical Society
  \textbf{83} (2001), no.~1, 1--50.

\bibitem{kumchev2006weyl}
A.~V. Kumchev, \emph{On {W}eyl sums over primes and almost primes}, Michigan
  Math. J \textbf{54} (2006), no.~2, 243--268.

\bibitem{kumchev2013weyl}
\bysame, \emph{On {W}eyl sums over primes in short intervals}, Number
  theory---arithmetic in {S}hangri-{L}a, Ser. Number Theory Appl., vol.~8,
  World Sci. Publ., Hackensack, NJ, 2013, pp.~116--131.

\bibitem{liu2006exponential}
J.~Y. Liu, G.~S. L{\"u}, and T.~Zhan, \emph{Exponential sums over primes in
  short intervals}, Science in China Series A \textbf{49} (2006), no.~5,
  611--619.

\bibitem{liu1996estimation}
J.~Y. Liu and T.~Zhan, \emph{Estimation of exponential sums over primes in
  short intervals {II}}, Analytic Number Theory: Proceedings of a Conference in
  Honor of Heini Halberstam, Birkhauser, 1996, pp.~571--606.

\bibitem{liu1999estimation}
\bysame, \emph{Estimation of exponential sums over primes in short intervals
  {I}}, Monatshefte f{\"u}r Mathematik \textbf{127} (1999), no.~1, 27--41.

\bibitem{lv2007exponential}
G.~S. L{\"u} and H.~X. Lao, \emph{On exponential sums over primes in short
  intervals}, Monatshefte f{\"u}r Mathematik \textbf{151} (2007), no.~2,
  153--164.

\bibitem{vaughan1980recent}
R.~C. Vaughan, \emph{Recent work in additive prime number theory}, Proceedings
  of the {I}nternational {C}ongress of {M}athematicians ({H}elsinki, 1978),
  Acad. Sci. Fennica, Helsinki, 1980, pp.~389--394. \MR{562631 (81j:10077)}

\bibitem{vinogradov1939estimation}
I.~M. Vinogradov, \emph{Estimation of certain trigonometric sums with prime
  variables}, Izv Acada Nauk SSSR \textbf{3} (1939), 371--398.

\bibitem{wei2014sums}
B.~Wei and T.~D. Wooley, \emph{On sums of powers of almost equal primes}, Proc.
  London Math. Soc. (to appear), 38pp. (2014).

\bibitem{zhan1991representation}
T.~Zhan, \emph{On the representation of large odd integer as a sum of three
  almost equal primes}, Acta Mathematica Sinica \textbf{7} (1991), no.~3,
  259--272.

\end{thebibliography}

\end{document}